\documentclass[notitlepage,12pt]{article}%
\usepackage{amsfonts}
\usepackage{amsmath}
\usepackage{amssymb}
\usepackage{graphicx}
\usepackage{color}%
\providecommand{\U}[1]{\protect\rule{.1in}{.1in}}
\newtheorem{theorem}{Theorem}

\newtheorem{definition}[theorem]{Definition}

\newtheorem{proposition}[theorem]{Proposition}
\newtheorem{remark}[theorem]{Remark}

\newenvironment{proof}[1][Proof]{\noindent\textbf{#1.} }{\ \rule{0.5em}{0.5em}}
\begin{document}

\title{Traveling and standing waves in coupled pendula and Newton's cradle}
\author{Carlos Garc\'{\i}a-Azpeitia\thanks{{\small Departamento de Matem\'{a}ticas,
Facultad de Ciencias, Universidad Nacional Aut\'{o}noma de M\'{e}xico, 04510
M\'{e}xico DF, M\'{e}xico. cgazpe@ciencias.unam.mx}}}
\maketitle

\begin{abstract}
The existence of traveling and standing waves is investigated for chains of
coupled pendula with periodic boundary conditions. The results are proven by
applying topological methods to subspaces of symmetric solutions. The main
advantage of this approach comes from the fact that only properties of the
linearized forces are required. This allows to cover a wide range of models
such as Newton's cradle, the Fermi-Pasta-Ulam lattice and the Toda lattice.


\end{abstract}

A chain of coupled pendula is a device of hanging limbs from an elastic rod.
When this device is set in motion, each limb behaves like a pendulum that
interacts with its neighbors by torsion forces. When the forces are
approximated by Hooke's law, the equations describing the pendula are
equivalent to the discretized Sine-Gordon equations.

Another device of interest is the Newton's Cradle. This system consists of
beads suspended by inelastic strings. In the absence of contact between beads,
they have pendular motion; when beads collide with their neighbors, they repel
each other with a Hertz's type force.

The paper aim is to study a model that includes the mentioned pendula among
other situations of interest, such as the Klein-Gordon, Fermi-Pasta-Ulam and
Toda lattices.

The movement of $n$ coupled oscillators, $q_{j}(t)\in\mathbb{R}$ for
$j=1,...,n$, with periodic boundary conditions, $q_{j}=q_{j+n}$, is described
by equations%
\begin{equation}
-\ddot{q}_{j}=U^{\prime}(q_{j})+W^{\prime}(q_{j}-q_{j-1})-W^{\prime}%
(q_{j+1}-q_{j})\text{,} \label{Ec}%
\end{equation}
where potentials $U$ and $W$ represent the dynamic and nonlinear interaction
of the oscillators, respectively.

The dynamic of a pendulum is governed by
\begin{equation}
U(x)=\omega^{2}(1-\cos x)\text{,\qquad}\omega^{2}=cg/l\text{,} \label{U}%
\end{equation}
where $g$ is the acceleration due to gravity and $l$ is the length of the
pendulum. The constant $c$ represents the coupling strength after a rescaling,
i.e. the normalized Hooke's Law is given by%
\[
W(x)=\frac{1}{2}x^{2}\text{,}%
\]
and the Hertz's contact force by%
\begin{equation}
W(x)=\frac{2}{5}\left\vert x\right\vert ^{5/2}\text{ if }x\leq0\text{, }\quad
W(x)=0\text{ if }x>0\text{.} \label{W}%
\end{equation}

We assume that equations (\ref{Ec}) have homogenous equilibria of the form%
\[
q_{j}(t)=a\text{\qquad for }j=1,..,n\text{.}%
\]
These properties hold true in the coupled pendula and the Newton's cradle when
$a=0$ and $a=\pi$. Under these considerations, the nonlinear equations
(\ref{Ec}) have periodic orbits arising from the homogenous equilibria.

The study of periodic orbits in Hamiltonian systems goes back to Poincare.
Lyapunov proves the nonlinear continuation of periodic orbits from normal
modes of elliptic equilibria under non-resonant conditions. Later on, the
Lyapunov center theorem was extended to consider the multiplicity of periodic
orbits and the global properties of the families. Regarding multiplicity, the
Weinstein--Moser theorem (1973) proves multiple periodic solutions of
Hamiltonian systems with fixed energy, while the Fadell--Rabinowitz theorem
(1978) considers fixed period. In \cite{Al78}, Alexande-Yorke (1978) prove the
global property of Lyapunov families. The proofs of these theorems make use of
topological invariants that consider the $S^{1}$-symmetry induced by time
translations. These theorems have been generalized to consider spatio-temporal
symmetries, for instance, see \cite{Bar93,Mo88} for fixed energy,
\cite{Bar93,FaRa77} for fixed period and \cite{BaKr10,Fi88,IzVi03} for global
bifurcation. Other equivariant approaches can be found in
\cite{Ch03,GoSc86,Va82} and references therein.

In the context of equation (\ref{Ec}), in \cite{Gi86}, local bifurcation of
periodic solutions is proven in the case $U=0$ and $W^{\prime\prime}(0)\neq0$,
and in \cite{GoSt02}, a Weinstein-Moser theorem is proven in the case
$U^{\prime\prime}=1$ and $W^{\prime\prime}(0)\neq0$. The present paper proves
the bifurcation of periodic solutions using Brouwer degree in spaces of
spatio-temporal symmetric functions, including a $\mathbb{Z}_{2}$-symmetry
induced by the reflection in time. In contrast with the results in \cite{Gi86}
and \cite{GoSt02}, this procedure has the advantage that allows to prove the
global property.

\emph{Theorem \ref{Thm1}.} Assume $W^{\prime\prime}(0)\neq0$, which is the
case in the coupled pendula. For each $k\in\lbrack1,n/2)\cap\mathbb{N}$ such
that
\begin{equation}
\nu_{k}=\sqrt{U^{\prime\prime}(a)+(2\sin k\pi/n)^{2}W^{\prime\prime}(0)}>0
\label{v}%
\end{equation}
is non-resonant (Definition \ref{nonres}), the homogenous equilibrium has
three global branches of $2\pi/\nu$-periodic solutions; for $k\in\{n/2,n\}$,
only one branch exists. The frequencies $\nu$ along the branches converge to
$\nu_{k}$ as the solutions approach the equilibrium and the bifurcating branch
is a continuum that either goes to infinity in Sobolev norm or period, or ends
at other bifurcation point.

In the coupled pendula, at the equilibrium $a=0$, the non-resonant condition
of $\nu_{k}$ holds true except for a finite of parameters $\omega$ (Section
4.1). For the resonant parameters the theorem proves only the existence of the
branches with the higher frequency $\nu_{j}=l\nu_{k}$.

When $W^{\prime\prime}(0)=0$ all frequencies are resonant,%
\[
\nu_{k}:=\sqrt{U^{\prime\prime}(a)}\text{\qquad for }k\in\{1,...,n\}\text{.}%
\]
These resonances make impossible to obtain multiple periodic solutions by
means of topological degree; instead, the existence of multiple standing waves
is proven with the Fadell-Rabinowitz theorem for odd potentials given in
\cite{FaRa77}.

\emph{Theorem \ref{Thm2}.}\ Assume $W^{\prime\prime}(0)=0$ and $U^{\prime
\prime}(a)>0$, which is the case in Newton's cradle. The homogenous
equilibrium has at least $n/2-1$ bifurcations of $2\pi/\nu$-periodic solutions
with symmetries (\ref{SI}) and (\ref{SII}). The frequency $\nu$ of the
periodic solutions is arbitrarily close to $\nu_{0}$, but the bifurcation does
not necessarily form a local continuum.

The periodic solutions have the symmetries of traveling and standing waves.
This fact is proven exploiting the equivariance of equations (\ref{Ec}) under
the action of the group
\[
D_{n}\times O(2),
\]
where $D_{n}$ is composed by permutations of the oscillators and $O(2)$ shifts
and reflects time; see Definition \ref{Def}.

The symmetries presented in Section 3 are valid along the global branches,
while the estimates shown in Section 2.1 are valid locally. We reproduce here
the simple case $n$ odd and $k=1$. The symmetries for $k=1$ have been analyzed
previously in \cite{Gi86} and \cite{GoSc86}. The present paper completes the
classification of the symmetries for all $k$'s.

\emph{Symmetries and estimates. }Let
\[
\zeta=2\pi/n\text{,}\qquad q_{j}(t)=a+x_{j}(\nu t)\text{,}%
\]
where $x_{j}(t)$ is $2\pi$-periodic and $\nu$ is the frequency. The branch of
traveling waves has symmetries and local estimates,%
\begin{align}
x_{j}(t)  &  =x_{n-j}(-t)=x_{j+1}(t-\zeta)\text{,}\label{T1}\\
x_{j}(t)  &  =r\cos(t+j\zeta)+\mathcal{O}(r^{2})\text{,}\nonumber
\end{align}
where $r$ is a parameterization of the branch and $\mathcal{O}(r^{2})$ is a
$2\pi$-periodic function. The branch of standing waves has symmetries and
local estimates,%
\begin{align}
x_{j}(t)  &  =x_{n-j}(t)=x_{j}(-t)\text{,}\label{SI}\\
x_{j}(t)  &  =r\cos(j\zeta)\cos t+\mathcal{O}(r^{2})\text{,}\nonumber
\end{align}
and the other,%
\begin{align}
x_{j}(t+\pi)  &  =x_{n-j}(t)=x_{j}(-t)\text{,}\label{SII}\\
x_{j}(t)  &  =r\sin(j\zeta)\sin t+\mathcal{O}(r^{2})\text{.}\nonumber
\end{align}

Traveling waves for Newton's cradle have been estimated asymptotically in
\cite{Ja11}, and for beads in \cite{Ja12} and \cite{StKe12}. In these papers,
traveling waves are constructed by means of a reduction to a single equation
with delay. This procedure is commonly used in many problems; see \cite{Pa05}
and the references therein. However, the reduction to one equation cannot be
used to prove existence of standing waves, and then, one of the achievements
of the present paper is the construction of them for the Newton's cradle.
Neither the Weinstein--Moser theorem in \cite{GoSt02} is applicable to the
Newtons's cradle because it assumes non-resonant conditions over
$W^{\prime\prime}(0)$. Through Proposition \ref{WM}, a Weinstein--Moser
theorem can be proven for the Newtons's cradle.

In the case of beads, $W^{\prime\prime}(0)=0$ and $U=0$, the topological
approach used to establish the existence of standing waves cannot be used due
to the fact that $\nu_{k}=0$ for $k=1,...,n$. In \cite{Ja12} is shown that
standing waves exist even for homogenous potential $W$ with $\nu_{k}=0$.
Therefore, further work is necessary to investigate the existence of standing
waves in the beads problem.

In Section 1, we set the bifurcation problem and make a global reduction to a
finite number of Fourier components. In Section 2.1, we prove the global
bifurcation in the case $W^{\prime\prime}(0)\neq0$. In Section 2.2, we prove
existence of standing waves in the case $W^{\prime\prime}(0)=0$ and
$U^{\prime\prime}(a)>0$. In Section 3, we describe the symmetries. In Section
4, we apply the theorems to the coupled pendula and the FPU and Toda lattices.
In Section 4.3, we present a comment about the existence of standing waves in
the case of a homogenous potential $W$ with $\nu_{k}=0$.

\section{Setting up the problem}

Let $q=(q_{1},...,q_{n})$ and%
\begin{equation}
V(q)=\sum_{j=1}^{n}[U(q_{j})+W(q_{j}-q_{j-1})]\text{.} \label{pot}%
\end{equation}
Equations (\ref{Ec}) can be expressed in vectorial form as%
\[
-\ddot{q}=\nabla V(q).
\]

Hereafter, we assume that the potential $V$ is twice differentiable and
\[
\mathbf{a}=(a,....,a)
\]
is an equilibrium, $\nabla V(\mathbf{a})=0$. This is equivalent to assume that
$U^{\prime}(a)=0$ and $W^{\prime}(0)=0$.

Using the change of variables $q(t)=\mathbf{a}+x(\nu t)$, the system of
equations become
\[
-\nu^{2}\ddot{x}=\nabla V(\mathbf{a}+x)\text{.}%
\]
Let $H_{2\pi}^{2}(\mathbb{R}^{n})$ be the Sobolev space of $2\pi$-periodic
functions.\ We define the operator $f$ from $H_{2\pi}^{2}(\mathbb{R}^{n})$ in
$L_{2\pi}^{2}(\mathbb{R}^{n})$ as%
\begin{equation}
f(x;\nu)=-\nu^{2}\ddot{x}-\nabla V(\mathbf{a}+x)\text{.}%
\end{equation}
Since $\mathbf{a}$ is an equilibrium, then $f(0;\nu)=0$ for all $\nu$.
Therefore, the branches of $2\pi/\nu$-periodic solutions emanating from the
equilibrium $\mathbf{a}$ correspond to zeros of $f(x;\nu)$ bifurcating from
$(0,\nu_{0})$.

\begin{definition}
\label{Def} Let $D_{n}$ be the subgroup of permutations generated by
\[
\zeta(j)=j+1,\qquad\kappa(j)=n-j\qquad\text{modulus }n\text{.}%
\]
Let $\rho:D_{n}\times O(2)\rightarrow GL(L_{2\pi}^{2})$ be the homomorphism
generated by
\begin{equation}
\rho(\gamma)(x_{1},...,x_{n})=(x_{\gamma(1)},...,x_{\gamma(n)})\text{,}%
\end{equation}
for $\gamma\in D_{n}$ and for $\varphi,\bar{\kappa}\in O(2)$,
\begin{equation}
\rho(\varphi)x(t)=x(t+\varphi)\text{,}\quad\rho(\bar{\kappa}%
)x(t)=x(-t)\text{.}%
\end{equation}
Then $\rho$ defines a $D_{n}\times O(2)$-representation of $L_{2\pi}%
^{2}(\mathbb{R}^{n})$ and induces the left action in $L_{2\pi}^{2}%
(\mathbb{R}^{n})$ given by $\rho(\gamma,x)=\rho(\gamma)x$ .
\end{definition}

Since $V(x)$ is invariant by the action of $D_{n}$, then $\nabla V(x)$ is
$D_{n}$-equivariant. Thus, the operator $f(x)$ is $D_{n}$-equivariant. Given
that the equations are autonomous and reversible in time, then $f(x)$ is
$D_{n}\times O(2)$-equivariant.

\subsection{Lyapunov-Schmidt reduction}

In the case $W^{\prime\prime}(0)\neq0$, we will prove existence of periodic
solutions using a global Lyapunov-Schmidt reduction and Brouwer degree.
Although an application of Leray--Schauder degree can provide similar results
without reductions, we prefer this approach because the reduction is required
in the Newton's cradle anyway. The idea of the global Lyapunov-Schmidt
reduction is taken from \cite{IzVi03}.

The Fourier expansion of $x\in L_{2\pi}^{2}$ and the projection $P$ are
defined as
\[
x(t)=\sum_{l\in\mathbb{Z}}x_{l}e^{il\tau}\text{ and }Px=\sum_{\left\vert
l\right\vert \leq l_{0}}x_{l}e^{ilt}\text{.}%
\]
Let $\mathbf{x}_{1}$\ and $\mathbf{x}_{2}$\ be the components of $x$, given
by
\[
\mathbf{x}_{1}=Px,\qquad\mathbf{x}_{2}=(I-P)x\text{.}%
\]
The components of $f$ are given by $f_{1}=Pf$ and $f_{2}=(I-P)f$.

We realize the global Lyapunov-Schmidt reduction in the set $\Omega_{\rho
}\times\Lambda_{\varepsilon}$, where%
\[
\Omega_{\rho}=\{x\in H_{2\pi}^{2}:\left\Vert x\right\Vert _{H_{2\pi}^{2}}%
<\rho\},\qquad\Lambda_{\varepsilon}=\{\nu>\varepsilon\}\text{.}%
\]

\begin{proposition}
There is a $l_{0}$ such that $\mathbf{x}_{2}(\mathbf{x}_{1},\nu)$\ is the only
solution of $f_{2}(\mathbf{x}_{1}+\mathbf{x}_{2},\nu)=0$ in $\Omega_{\rho
}\times\Lambda_{\varepsilon}$. Thus $f(\mathbf{x}_{1}+\mathbf{x}_{2},\nu)=0$
if and only if $\phi(\mathbf{x}_{1},\nu)=0$, where%
\begin{align}
\phi(\mathbf{x}_{1},\nu)  &  =f_{1}(\mathbf{x}_{1}+\mathbf{x}_{2}%
(\mathbf{x}_{1},\nu);\nu)\\
&  =-\nu^{2}\partial_{tt}\mathbf{x}_{1}-P\nabla V(\mathbf{a}+\mathbf{x}%
_{1}+\mathbf{x}_{2}(\mathbf{x}_{1},\nu))\text{.}\nonumber
\end{align}
Furthermore, the reduced map $\phi(\mathbf{x}_{1},\nu)$ is $D_{n}\times
O(2)$-equivariant, where the action of $\varphi,\bar{\kappa}\in O(2)$ in the
$l$-th Fourier component is given by
\[
\rho(\varphi)x_{l}=e^{il\varphi}x_{l},\qquad\rho(\bar{\kappa})x_{l}=\bar
{x}_{l}.
\]

\end{proposition}

\begin{proof}
If we find a positive constant $\alpha$ such that
\[
\left\Vert \partial_{\mathbf{x}_{2}}f_{2}(\mathbf{x}_{1}+\mathbf{x}_{2}%
)y_{2}\right\Vert _{L_{2\pi}^{2}}\geq\alpha\left\Vert y_{2}\right\Vert
_{H_{2\pi}^{2}}%
\]
for all $(x,\nu)\in\Omega_{\rho}\times\Lambda_{\varepsilon}$, the global
implicit function theorem due to Hadamard (Theorem 5.1.5 in \cite{Be77})
implies existence of a unique function $\mathbf{x}_{2}(\mathbf{x}_{1},\nu)$
such that $f_{2}(\mathbf{x}_{1}+\mathbf{x}_{2}(\nu,\mathbf{x}_{1}),\nu)=0$.
Using the uniqueness of $\mathbf{x}_{2}(\mathbf{x}_{1},\nu)$, it can be proven
that $f_{1}(\mathbf{x}_{1}+\mathbf{x}_{2}(\mathbf{x}_{1},\nu),\nu)$ is
$D_{n}\times O(2)$-equivariant.

Let $y_{2}\in(I-P)H_{2\pi}^{2}$. Since $V\in C^{2}(\mathbb{R}^{n})$ and
$\left\Vert x\right\Vert _{C_{2\pi}^{0}}<c\rho$ for all $x\in\Omega_{\rho}$,
then $\left\Vert D^{2}V(\mathbf{a}+x)\right\Vert _{C_{2\pi}^{0}}<C\rho$ for
all $x\in\Omega_{\rho}$. Therefore,%
\[
\left\Vert (I-P)D^{2}V(\mathbf{a}+x)y_{2}\right\Vert _{L_{2\pi}^{2}}\leq
C\rho\left\Vert y_{2}\right\Vert _{L_{2\pi}^{2}}\leq C(\rho/l_{0}%
^{2})\left\Vert y_{2}\right\Vert _{H_{2\pi}^{2}}\text{.}%
\]
Using the previous estimate and $\left\Vert \nu^{2}\partial_{tt}%
y_{2}\right\Vert _{L_{2\pi}^{2}}\geq\varepsilon^{2}\left\Vert y_{2}\right\Vert
_{H_{2\pi}^{2}}$ for $\nu\in\Lambda_{\varepsilon}$, we have%
\[
\left\Vert \partial_{\mathbf{x}_{2}}f_{2}y_{2}\right\Vert _{L_{2\pi}^{2}}%
\geq\left\Vert \nu^{2}\partial_{tt}y_{2}\right\Vert _{L_{2\pi}^{2}}-\left\Vert
(I-P)D^{2}V(\mathbf{a}+x)y_{2}\right\Vert _{L_{2\pi}^{2}}\geq\alpha\left\Vert
y_{2}\right\Vert _{H_{2\pi}^{2}}\text{,}%
\]
where $\alpha=\varepsilon^{2}-C(\rho/l_{0}^{2})$. We conclude that $\alpha>0$
if $l_{0}>C\sqrt{\rho}/\varepsilon$. Note that the number of Fourier
components $l_{0}$ goes to infinity as $\varepsilon\rightarrow0$ and
$\rho\rightarrow\infty$.
\end{proof}

Using Taylor's expansion of the map $f_{2}(\mathbf{x}_{1}+\mathbf{x}_{2})$, we
obtain the estimate
\begin{equation}
\left\Vert \mathbf{x}_{2}(\mathbf{x}_{1},\nu)\right\Vert _{H_{2\pi}^{2}}\leq
c\left\Vert \mathbf{x}_{1}\right\Vert ^{2} \label{Est}%
\end{equation}
for $x$ close to $0$. Thus, the linearization of the reduced map at $(0,\nu)$
is
\begin{equation}
\phi^{\prime}(0;\nu)\mathbf{x}_{1}=-\nu^{2}\partial_{tt}\mathbf{x}_{1}%
-D^{2}V(\mathbf{a})\mathbf{x}_{1}=\sum_{\left\vert l\right\vert \leq l_{0}%
}M(l\nu)x_{l}\text{,}%
\end{equation}
where
\[
M(l\nu)=(l\nu)^{2}I-D^{2}V(\mathbf{a}).
\]

\subsection{Irreducible representations}

In this section we identify the irreducible representations of $O(2)\times
D_{n}$. In the $l$-th Fourier component, the action of $O(2)$ is%
\[
\rho(\varphi)x_{l}=e^{li\varphi}x_{l},\qquad\rho(\bar{\kappa})x_{l}=\bar
{x}_{l}\text{.}%
\]
The Fourier components are subrepresentations of the group $O(2)$; then, we
need to find the irreducible representations of $x_{l}\in\mathbb{C}^{n}$ under
the action of $D_{n}$.

Let $\zeta=2\pi/n$ and $e_{k}\in\mathbb{C}^{n}$ be
\begin{equation}
e_{k}=n^{-1/2}(e^{1(ik\zeta)},e^{2(ik\zeta)},...,e^{n(ik\zeta)})\text{.}%
\end{equation}
The vectors $e_{k}$ for $k=1,...,n$ are orthonormal and their direct sum is
the whole space $\mathbb{C}^{n}$. Therefore, we can expand $x_{l}=\sum
_{k\in\mathbb{Z}_{n}}x_{k,l}$, where $\mathbb{Z}_{n}=\{1,...,n\}$, and%
\[
x(t)=\sum_{(k,l)\in\mathbb{Z}_{n}\times\mathbb{Z}}x_{k,l}e_{k}e^{ilt}\text{.}%
\]

\begin{proposition}
\label{Ac}For $k=n/2,n$ the action of the group $D_{n}\times O(2)$ in
$x_{k,1}\in\mathbb{C}$ is given by
\begin{equation}
\rho(\zeta,\varphi)x_{k,1}=\pm e^{i\varphi}x_{k,1}\text{,}\quad\rho
(\kappa)x_{k,1}=x_{k,1}\text{,}\quad\rho(\bar{\kappa})x_{k,1}=\bar{x}%
_{k,1}\text{,} \label{AcR}%
\end{equation}
with negative sign for $k=n/2$. For $k\in\lbrack1,n/2)\cap\mathbb{N}$, the
action in $(x_{k,1},x_{n-k,1})\in\mathbb{C}^{2}$\ is given by%
\begin{align}
\rho(\zeta,\varphi)(x_{k,1},x_{n-k,1})  &  =e^{i\varphi}(e^{ik\zeta}%
x_{k,1},e^{-ik\zeta}x_{n-k,1})\text{, }\\
\rho(\kappa)(x_{k,1},x_{n-k,1})  &  =(x_{n-k,1},x_{k,1})\text{, }\nonumber\\
\rho(\bar{\kappa})(x_{k,1},x_{n-k,1})  &  =(\bar{x}_{n-k,1},,\bar{x}%
_{k,1})\text{.}\nonumber
\end{align}

\end{proposition}

\begin{proof}
The actions of $\zeta$ and $\kappa$ in $e_{k}$ are
\begin{align*}
\rho(\zeta)e_{k}  &  =n^{-1/2}(e^{2(ik\zeta)},e^{3(ik\zeta)},...,e^{n(ik\zeta
)},e^{(ik\zeta)})=e^{ik\zeta}e_{k}\text{,}\\
\rho(\kappa)e_{k}  &  =(e^{(n-1)ik\zeta},e^{(n-2)ik\zeta},...,e^{2ik\zeta
},e^{1ik\zeta},e^{nik\zeta})=e_{n-k}\text{.}%
\end{align*}
Moreover, the action of $\bar{\kappa}\in O(2)\ $is
\[
\rho(\bar{\kappa})e_{k}z=n^{-1/2}(e^{-ik\zeta}\bar{z},e^{-2ik\zeta}\bar
{z},...,e^{-(n-1)ik\zeta}\bar{z},e^{-nik\zeta}\bar{z})=e_{n-k}\bar{z}\text{.}%
\]
The result follows.
\end{proof}

By the previous proposition, the subspaces generated by $e_{n}$, $e_{n/2}$ and
$e_{k}\oplus e_{n-k}$\ for $k\in\lbrack1,n/2)\cap\mathbb{N}$ are irreducible representations.

\subsection{Isotropy groups}

In this section, we identify the maximal isotropy groups of the irreducible
representations presented in Proposition \ref{Ac}.

The action of $D_{n}\times O(2)$ in the subspaces generated by $e_{n}$ and
$e_{n/2}$ are given by (\ref{AcR}), where the minus sign corresponds to the
representation generated by $e_{n/2}$. The subspace $x_{n,1}\in\mathbb{R}$ has
isotropy group
\begin{equation}
T_{n}=\left\langle (\zeta,0),(\kappa,0),(0,\bar{\kappa})\right\rangle \text{,}%
\end{equation}
and the subspace$\mathbb{\ }x_{n/2,1}\in\mathbb{R}$,%
\begin{equation}
T_{n/2}=\left\langle (\zeta,\pi),(\kappa,0),(0,\bar{\kappa})\right\rangle
\text{.}%
\end{equation}
Both isotropy groups $T_{k}$ for $k=n,n/2$ have fixed point spaces of
dimension one.

\begin{definition}
Let $h$ be the maximum common divisor of $k$ and $n$,
\[
\bar{k}=\frac{k}{h}\text{ and }\bar{n}=\frac{n}{h}\text{.}%
\]

\end{definition}

The cases $k\in\lbrack1,n/2)\cap\mathbb{N}$ are analyzed in the appendix. We
reproduce here the main results.

\begin{proposition}
For $k\in\lbrack1,n/2)\cap\mathbb{N}$, the representation $e_{k}\oplus
e_{n-k}$ has three maximal isotropy groups with fixed point spaces of
dimension one. An isotropy group is
\[
T_{k}=\left\langle (\zeta,-k\zeta),(\kappa,\bar{\kappa}),(\bar{n}%
\zeta,0)\right\rangle \text{.}%
\]
For $\bar{n}$ odd, the other isotropy groups are%
\[
S_{k}=\left\langle (\kappa,0),(0,\bar{\kappa}),(\bar{n}\zeta,0)\right\rangle
\text{,\qquad}\widetilde{S}_{k}=\left\langle (\kappa,\pi),(0,\pi\bar{\kappa
}),(\bar{n}\zeta,0)\right\rangle \text{.}%
\]
The case $\bar{n}$ even is given in the appendix.
\end{proposition}

These isotropy groups are relevant because the reduced map
\[
\phi^{H}(x;\nu):\mathrm{Fix}(H)\rightarrow\mathrm{Fix}(H)
\]
is well defined and under non-resonant conditions, the linearization
$D\phi^{H}(0,\nu)$ has a simple eigenvalue in the kernel.

\subsection{Linearization}

The Hessian of $V$ is%
\[
D^{2}V(\mathbf{a})=U^{\prime\prime}(a)I+W^{\prime\prime}(0)A\text{,}%
\]
where $A=(a_{i,j})_{i,j=1}^{n}$ is the matrix defined by $a_{i,j}=2$ if $i=j$,
$a_{i,j}=-1$ if $\left\vert i-j\right\vert =1$ modulus $n$, and $a_{i,j}=0$ otherwise.

\begin{proposition}
The matrix $M(\nu)=\nu^{2}I-D^{2}V(\mathbf{a})$ is diagonal in the basis
$\{e_{k}\}$,
\begin{equation}
x=\sum_{k\in\mathbb{Z}_{n}}x_{k}e_{k},\qquad M(\nu)x=\sum_{k\in\mathbb{Z}_{n}%
}\lambda_{k}(\nu)x_{k}e_{k}\text{,}%
\end{equation}
where the eigenvalues are
\begin{equation}
\lambda_{k}(\nu)=\nu^{2}-U^{\prime\prime}(a)-(2\sin k\zeta/2)^{2}%
W^{\prime\prime}(0).
\end{equation}
\ 
\end{proposition}

\begin{proof}
Since the $j$-th coordinate of $e_{k}$ is $n^{-1/2}e^{ij\zeta}$, then the
$j$-th coordinate of $Ae_{k}$ is
\[
(2-(e^{-ik\zeta}+e^{ik\zeta}))n^{-1/2}e^{ij\zeta}=4\sin^{2}(k\zeta
/2)n^{-1/2}e^{ij\zeta}\text{.}%
\]
We conclude that $Ae_{k}=4\sin^{2}(k\zeta/2)e_{k}$ and%
\[
M(\nu)e_{k}=[\nu^{2}I-U^{\prime\prime}(a)]e_{k}-W^{\prime\prime}%
(0)Ae_{k}=\lambda_{k}(\nu)e_{k}\text{.}%
\]

\end{proof}

The fact that $\lambda_{n-k}=\lambda_{k}$ for $k\in\lbrack1,n/2)\cap
\mathbb{N}$ is consequence of Schur's lemma. That is, since $e_{k}\oplus
e_{n-k}$ is an irreducible representation, Schur's lemma implies that
$M(\nu)=\lambda I$ in $e_{k}\oplus e_{n-k}$.

In the basis $x_{k,l}$, we have
\begin{equation}
\phi^{\prime}(0;\nu)\mathbf{x}_{1}=\sum_{\left\vert l\right\vert \leq l_{0}%
}\sum_{k\in\mathbb{Z}_{n}}\lambda_{k}(l\nu)x_{k,l}e_{k}e^{ilt}\text{.}%
\end{equation}

\section{Main results:\ Bifurcation theorems}

We proceed to prove two bifurcation theorems.

\subsection{Coupled pendula}

If $W^{\prime\prime}(0)\neq0$, we avoid resonant frequencies by assuming that
$D^{2}V(\mathbf{a})$ is invertible, which holds if $U^{\prime\prime}(a)$ and
$W^{\prime\prime}(0)$ are positive.\ 

\begin{definition}
\label{nonres}We say that the frequency
\begin{equation}
\nu_{k}=\sqrt{U^{\prime\prime}(a)+(2\sin k\pi/n)^{2}W^{\prime\prime}(0)}%
\end{equation}
is non-resonant if $l\nu_{k}\neq\nu_{j}$ for $j\in(k,n/2]\cap\mathbb{N}$ and
$l\geq2$.
\end{definition}

\begin{theorem}
\label{Thm1} Assume $W^{\prime\prime}(0)\neq0$ and $D^{2}V(\mathbf{a})$ is
invertible. For each $k\in\lbrack0,n/2]\cap\mathbb{N}$ such that $\nu_{k}>0$
is non-resonant, the equilibrium $\mathbf{a}$ has three global bifurcations of
$2\pi/\nu$-periodic solutions emanating from $\nu=\nu_{k}$ with isotropy
groups $T_{k}$, $S_{k}$, and $\widetilde{S}_{k}$.
\end{theorem}

\begin{proof}
The linear map $\phi^{\prime}(0;\nu)$ has $l$-th Fourier block $M(l\nu)$ with
eigenvalues
\[
\lambda_{j}(l\nu)=(l\nu)^{2}-\nu_{j}^{2}\text{,}%
\]
for $j\in\lbrack0,n/2]\cap\mathbb{N}$.

For $l=0$, the matrix $M(0)=-D^{2}V(\mathbf{a})$ is invertible by hypothesis.
For $l\geq2$, since $W^{\prime\prime}(0)\neq0$, frequencies $\nu_{j}$ are
increasing in $j$. Thus $l\nu_{k}>\nu_{j}$ for $j\in\lbrack0,k]\cap\mathbb{N}$
and, by hypothesis, $l\nu_{k}\neq\nu_{j}$ for $j\in(k,n/2]\cap\mathbb{N}$.
Therefore, matrices $M(l\nu_{k})$ are invertible for $l\geq2$. Given that
$M(l\nu)$ is continuous in $\nu$, the blocks $M(l\nu)$ are invertible for
$\nu$ close to $\nu_{k}$ .

For $l=1$, since $\nu_{k}$ is increasing in $k\in\lbrack0,n/2]\cap\mathbb{N}$,
then $\nu_{k}^{2}-\nu_{j}^{2}\neq0$ for $j\neq k$. Therefore, the
linearization $\phi^{\prime}(0;\nu)$ is invertible for $\nu$ close to $\nu
_{k}$, except for the block that corresponds to the representation
$(x_{k,1},x_{n-k,1})\in e_{k}\oplus e_{n-k}$,%
\begin{equation}
\lambda_{k}(\nu)I:\mathbb{C}^{2}\rightarrow\mathbb{C}^{2}\text{.} \label{B}%
\end{equation}

Set $H$ equal to $T_{k}$, $S_{k}$ or $\widetilde{S}_{k}$. Since the group $H$
has fixed point spaces of real dimension equal to one in $e_{k}\oplus e_{n-k}%
$, the restriction $D\phi^{H}(0;\nu_{k})$ has a simple eigenvalue crossing
zero in (\ref{B}). Using Brouwer degree as in Theorem 14 in \cite{GaIz11} or
\cite{IzVi03}, we conclude existence of a local bifurcation in the fixed point
space of $H$.

The global property follows from assuming that the branch is contained in the
set $\Omega_{\rho}^{H}\times\Lambda_{\varepsilon}$ for some $\varepsilon$ and
$\rho$, unless it is an unbounded continuum set with period or Sobolev norm
going to infinite. Applying Brouwer degree to the reduction in $\Omega_{\rho
}^{H}\times\Lambda_{\varepsilon}$, we conclude that the sum of the local
degrees at the bifurcation points is zero, as in Theorem 5.2 in \cite{IzVi03}
or Theorem 15 in \cite{GaIz11}.
\end{proof}

\begin{remark}
In the case $U=0$, the Hessian $D^{2}V(\mathbf{a})$ has a zero-eigenvalue
corresponding to the conserved quantity $\sum_{j=1}^{n}q_{j}=0$. If this is
the only zero-eigenvalue, we can extended the previous theorem using the
restriction of $f$ to the subspace%
\[
X=\{x\in L_{2\pi}^{2}:\sum_{j=1}^{n}x_{j}=0\}\text{.}%
\]
Also, analogous theorems can be proven in the case of $D_{n}$-equivariant long
range interactions, $W(q_{1},...,q_{n})$.
\end{remark}

Under non-resonant considerations, the local reduction can be realized on the
$1$-th Fourier component
\[
\mathbf{x}_{1}=\sum_{\left\vert l\right\vert =1}x_{l}e^{ilt}\text{.}%
\]

Let $r$ be a parameterization of the amplitude of the local branch, where the
frequency $\nu$ is a functions of $r$. The branch with isotropy group $T_{k}$
has eigenvalues corresponding to the coordinates $(x_{k,1},x_{n-k,1})=(r,0)$
(see the appendix). We conclude that the projected component $\mathbf{x}_{1}$
can be estimated by
\[
\mathbf{x}_{1}(t)=r(e^{it}e_{k}+\overline{e^{it}e_{k}})+\mathcal{O}%
(r^{2})\text{.}%
\]
Moreover, since $e^{it}e_{k}+\overline{e^{it}e_{k}}=2\cos(t+jk\zeta)$ and
$\mathbf{x}_{2}(\mathbf{x}_{1};\nu)=O(r^{2})$, then%
\begin{equation}
x_{j}(t)=2r\cos(t+jk\zeta)+\mathcal{O}(r^{2})\text{,}%
\end{equation}
where $\mathcal{O}(r^{2})$ is a $2\pi$-periodic function of order $r^{2}$.

Similarly, the coordinates for the isotropy group $S_{k}$ are $(x_{k,1}%
,x_{n-k,1})=(r,r)$, then%
\[
\mathbf{x}_{1}(t)=r(e^{it}e_{k}+\overline{e^{it}e_{k}}+e^{it}e_{n-k}%
+\overline{e^{it}e_{n-k}})+\mathcal{O}(r^{2})\text{.}%
\]
We conclude
\begin{equation}
x_{j}(t)=4r\cos jk\zeta\cos t+\mathcal{O}(r^{2})\text{.}%
\end{equation}

For the group $\widetilde{S}_{k}$ and $n$ odd, the coordinates are
$(x_{k,1},x_{n-k,1})=(r,-r)$, then%
\begin{equation}
x_{j}(t)=-4r\sin jk\zeta\sin t+\mathcal{O}(r^{2})\text{.}%
\end{equation}
For $n$ even, the coordinates are $(x_{k,1},x_{n-k,1})=(r,re^{i\zeta})$, then%
\begin{equation}
x_{j}(t)=-4r\sin(jk\zeta-\zeta/2)\sin(t+\zeta/2)+\mathcal{O}(r^{2})\text{.}%
\end{equation}

Note that for $n$ odd, since $\sin(2\pi j/n)=0$ for $j\in\{n,n/2\}\cap
\mathbb{N}$ and $\cos(2\pi j/n)\neq0$, standing waves $S_{k}$ and
$\widetilde{S}_{k}$ have nodes of different orders, $r$ and $r^{2}$ respectively.

\subsection{Newtons's cradle}

For Newtons's cradle, $W$ is given by (\ref{W}) and $U(x)=\omega^{2}(1-\cos
x)$. Since $\nu_{k}=\omega$ for $k=1,..,n$, the previous theorem cannot
provide the existence of many periodic solutions. In this case, the
application of Weinstein--Moser theorem guaranties the existence of at least
$n$ periodic solutions in each constant energy surface,%
\begin{equation}
H(q,p)=\frac{1}{2}\left\Vert p\right\Vert ^{2}+V(q)\text{.}\label{Ham}%
\end{equation}
However, these solutions may agree with the $n$ traveling waves found in
\cite{Ja11}, and the symmetries have to be considered in order to obtain new solutions.

\begin{definition}
Let $S=S^{1}$ and $\widetilde{S}$ be the group generated by $(\kappa,\pi)$
$(0,\pi\bar{\kappa})$ if $n$ is odd, and $(\kappa\zeta,0)$ and $(0,\zeta
\bar{\kappa})$ if $n$ is even.
\end{definition}

We will prove existence of standing waves using the fixed point spaces of the
groups $S$ and $\widetilde{S}$, which are the isotropy groups $S_{1}$ and
$\widetilde{S_{1}}$ without the generator $(\pi,\pi)$.

Since $\nu_{k}=\omega$ for all $k$, the $1$-Fourier component is non-resonant
with other Fourier components. Thus the local reduction of Section 1.1 can be
realized on the $1$-th Fourier component,%
\[
\mathbf{x}_{1}=x_{1}e^{it}+\bar{x}_{1}e^{-it}\text{,}%
\]
and the local bifurcation map is defined by $\phi(\mathbf{x}_{1}%
;\nu):\mathbb{C}^{n}\rightarrow\mathbb{C}^{n}$ for $(x,\nu)$ close to
$(0,\omega)$.

Since the operator $f(x)$ is the gradient of%
\[
F(x)=\int_{0}^{2\pi}\left(  \frac{\nu^{2}}{2}\left\vert \mathcal{\partial}%
_{t}x\right\vert ^{2}-V(x)\right)  dx\text{,}%
\]
we can conclude (see Section 1.9 in \cite{IzVi03}) that $\phi(\mathbf{x}_{1})$
is the gradient of the reduced potential%
\[
\Phi(\mathbf{x}_{1})=F(\mathbf{x}_{1}+\mathbf{x}_{2}(\mathbf{x}_{1}))\text{.}%
\]
That is, $\Phi^{\prime}(\mathbf{x}_{1})h=\left\langle \phi(\mathbf{x}%
_{1}),h\right\rangle _{L_{2\pi}^{2}}$.

Since Lyapunov-Schmidt reductions preserve equivariant properties, the
potential $\Phi(x_{1},\nu)$ is $D_{n}\times O(2)$-invariant. Setting $H$ equal
to $S$ or $\widetilde{S}$, this implies that the restriction $\Phi
^{H}:\mathrm{Fix}(H)\rightarrow\mathbb{R}$ is well defined and the gradient is%
\[
\nabla\Phi^{H}=\phi^{H}:\mathrm{Fix}(H)\rightarrow\mathrm{Fix}(H)\text{.}%
\]
Furthermore, the element $\pi\in O(2)$ is contained in the Weyl group of $H$
and acts multiplying by $-1$ the fixed point space of $H$ (see appendix), then
$\Phi^{H}(x_{1})=\Phi^{H}(-x_{1})$.

Using $\mathbf{a}=0$, $V(0)=0$, $D^{2}V(0)=\omega^{2}I$ and $\mathbf{x}%
_{2}(\mathbf{x}_{1};\nu)=\mathcal{O}(\left\vert \mathbf{x}_{1}\right\vert
^{2})$, we estimate
\begin{equation}
\Phi(x_{1},\nu)=2\pi(\nu^{2}-\omega^{2})\left\vert x_{1}\right\vert
^{2}+o(\left\vert x_{1}\right\vert ^{2})\text{.}%
\end{equation}
Therefore, the Fadell-Rabinowitz theorem in \cite{FaRa77} implies that the odd
potential $\Phi^{H}(\mathbf{x}_{1},\nu)$ has at least $\dim\mathrm{Fix}(H)$
branches of critical points. Since $\dim\mathrm{Fix}(H)\geq n/2-1$ for $H$
equal to $S$ and $\widetilde{S}$ (see the appendix), the following theorem holds.

\begin{theorem}
\label{Thm2}Set $H$ equal to $S$ or $\widetilde{S}$. If $(\mathbf{x}_{1}%
,\nu)=(0,\omega)\ $is an isolated point of the potential $\Phi^{H}%
(\mathbf{x}_{1},\nu)$, then $f(x,\nu)$ has at least $d_{-}$ zeros for
$\nu<\omega$ and $d_{+}$ for $\nu>\omega$ in the fixed point space of $H$. The
zeros converge to $(0,\omega)$ as $\nu\rightarrow\omega$ and
\[
d_{-}+d_{+}\geq n/2-1\text{.}%
\]

\end{theorem}

Therefore, there are at least $n/2-1$ bifurcations of periodic solutions with
symmetries $S$ and $\widetilde{S}$ emanating from the homogeneous equilibria.
The information about the symmetries of these $n/2-1$ bifurcations can be
improved if one considers similar procedures in the fixed point spaces of
$S_{k}$ and $\widetilde{S}_{k}$.

Another proof of the previous theorem can be given with $\mathbb{Z}_{2}%
$-equivariant Conley index. The idea in \cite{Bar93} consists on using the
invariant property of the Conley index and the change in the unstable set of
the gradient flow generated by $\Phi^{H}$, which has dimension $\dim
\mathrm{Fix}(H)$ for $\nu<\omega$ and $0$ for $\nu>\omega$.

Using Proposition \ref{WM} and Theorem 9.9 in \cite{Bar93}, we can prove that
the energy surface $H^{-1}(\varepsilon)$ for small $\varepsilon$ contains at
least $n/2-1$ solutions with isotropy groups $S$ and $\widetilde{S}$
(Weinstein--Moser theorem). These kinds of solutions are called brake orbits
in \cite{Bar93}. A Weinstein--Moser theorem for coupled pendula was proven in
\cite{GoSt02} using results of \cite{Mo88}.

\section{Description of symmetries}

The isotropy groups $T_{k}$, $S_{k}$ and $\widetilde{S}_{k}$ have the
generator $(\bar{n}\zeta,0)$. Functions fixed by $(\bar{n}\zeta,0)$ satisfy
\[
x_{j}(t)=x_{j+\bar{n}}(t)\text{,}%
\]
where $h$ is the maximum common divisor of $k$ and $n$, $\bar{k}=k/h$ and
$\bar{n}=n/h$.

Therefore, the oscillators in these solutions form a wave of length $\bar{n}$
that is repeated $h$ times along the pendula. In the following discussion, we
describe only the wave of length $\bar{n}$,%
\[
(x_{1},...,x_{\bar{n}})\text{.}%
\]

\subsection{Traveling waves}

The group $T_{k}$ has generators $(\zeta,-k\zeta)$ and $(\kappa,\bar{\kappa}%
)$, solutions with isotropy group $T_{k}$ have symmetries%
\begin{equation}
x_{j}(t)=x_{n-j}(-t)=x_{j+1}(t-\bar{k}(2\pi/\bar{n}))\text{,} \label{T}%
\end{equation}
For $\bar{k}=1$, two consecutive oscillators have a phase shift of $2\pi
/\bar{n}$,
\[
x_{j}(t)=x_{j+1}(t-2\pi/\bar{n})\text{.}%
\]
For $\bar{k}\neq1$, the solutions are just permutations of the case $\bar
{k}=1$.

\subsection{Standing waves of the first kind}

Standing waves have different behavior depending on the parity of $\bar{n}%
\ $and $\bar{n}/2$. We present three cases for each kind of standing waves.
Given that cases $k\neq1$ are permutations of $\bar{k}=1$, we present only the
case $\bar{k}=1$.

Hereafter, we denote%
\[
x_{j}^{\ast}(t)=x_{j}(t+\pi).
\]

\subsubsection*{$\bar{n}$ odd}

Since the isotropy group $S_{k}$ has generators $(\kappa,0)$ and $(0,\kappa)$,
functions fixed by $S_{k}$ have symmetries
\[
x_{j}(t)=x_{\bar{n}-j}(t)=x_{j}(-t)\text{.}%
\]
Setting $\bar{n}=2m+1$, the wave of length $\bar{n}$ is%
\begin{equation}
(x_{0},x_{1},...,x_{m},x_{m},...,x_{1})\text{,} \label{1SI}%
\end{equation}
where $x_{j}$ are even functions for $j=0,...,m$.

\subsubsection*{$\bar{n}/2$ odd$\ $}

The isotropy group $S_{k}$ has generators $(\kappa,0)$, $(\frac{\bar{n}}%
{2}\zeta,\pi)$ and $(0,\bar{\kappa})$, solutions in the fixed point space of
$S_{k}$ satisfy
\[
x_{j}(t)=x_{\bar{n}-j}(t)=x_{j+\bar{n}/2}^{\ast}(t)=x_{j}(-t)\text{.}%
\]
Setting $\bar{n}=4m+2$, the wave of length $\bar{n}$ is
\begin{equation}
(x_{0},x_{1},...,x_{m},x_{m}^{\ast},...,x_{1}^{\ast},x_{0}^{\ast},x_{1}^{\ast
},...,x_{m}^{\ast},x_{m},...,x_{1})\text{,} \label{2SI}%
\end{equation}
where $x_{j}$ are even functions for $j=0,...,m$.

\subsubsection*{$\bar{n}/2$ even}

Setting $\bar{n}=4m+4$, we conclude that the wave of length $\bar{n}$ is
\begin{equation}
(x_{0},x_{1},...,x_{m}=x_{m}^{\ast},...,x_{1}^{\ast},x_{0}^{\ast},x_{1}^{\ast
},...,x_{m}^{\ast}=x_{m},...,x_{1})\text{,} \label{3SI}%
\end{equation}
where $x_{j}$ are even functions for $j=0,...,m$ and $x_{m}\ $is $\pi$-periodic.

Observe that, although solutions (\ref{2SI}) and (\ref{3SI}) have the same
isotropy group, the solutions are qualitatively different at the oscillator
$x_{m}$.

\subsection{Standing waves of the second kind}

\subsubsection*{$\bar{n}$ odd}

The isotropy group $\widetilde{S}_{k}$ has generators $(\kappa,\pi)$ and
$(0,\pi\bar{\kappa})$. Functions fixed by $\tilde{S}_{k}$ have symmetries%
\[
x_{j}(t)=x_{\bar{n}-j}(t+\pi)=x_{j}(\pi-t).
\]
Setting $\bar{n}=2m+1$, the wave of length $\bar{n}$ is%
\begin{equation}
(x_{0},x_{1},...,x_{m},x_{m}^{\ast},...,x_{1}^{\ast})\text{,} \label{1SII}%
\end{equation}
where $x_{j}(\cdot+\pi/2)$ are even functions for $j=1,...,m$ and $x_{0}(t)$
is even $\pi$-periodic.

\subsubsection*{$\bar{n}/2$ odd$\ $}

Since the isotropy group $\widetilde{S}_{k}$ has generators $(\kappa\zeta,0)$,
$(\frac{\bar{n}}{2}\zeta,\pi)\ $and $(0,\zeta\bar{\kappa})$, solutions in the
fixed point space of $\widetilde{S}_{k}$ satisfy%
\[
x_{j}(t)=x_{\bar{n}-(j+1)}(t)=x_{\bar{n}/2+j}^{\ast}(t)=x_{j}(2\pi/n-t).
\]
Setting $\bar{n}=4m+2$, the wave of length $\bar{n}$ is%
\begin{equation}
(x_{0},x_{1},...,x_{m}=x_{m}^{\ast},...,x_{1}^{\ast},x_{0}^{\ast},x_{0}^{\ast
},x_{1}^{\ast},...,x_{m}^{\ast}=x_{m},...,x_{1},x_{0})\text{,} \label{2SII}%
\end{equation}
where $x_{j}(\pi/n+\cdot)$ are even for $j=0,...,m$ and $x_{m}$ is $\pi$-periodic.

\subsubsection*{$\bar{n}/2$ even}

For $\bar{n}=4m+4$, we conclude that solutions with isotropy group
$\widetilde{S}_{k}$ satisfy%
\begin{equation}
(x_{0},...,x_{m},x_{m}^{\ast},...,x_{0}^{\ast},x_{0}^{\ast},...,x_{m}^{\ast
},x_{m},...,x_{0}), \label{3SII}%
\end{equation}
where $x_{j}(\pi/n+\cdot)$ are even for $j=0,...,m$.

Solutions (\ref{2SII}) and (\ref{3SII}) are different close to the oscillator
$x_{m}$.

\section{Applications}

In most applications the potential $W$ is convex. If $U$ is also convex, the
frequencies $\nu_{k}$ are always positive; if $U$ is concave, the frequencies
$\nu_{k}$ are positive for $k\in\lbrack k_{0},n/2]\cap\mathbb{N}$.

\subsection{Coupled pendula}

The coupled pendula via torsion springs are governed by the dynamics of
$U(x)=\omega^{2}(1-\cos x)$ and $W(x)=x^{2}/2$. Since $U^{\prime}%
(0)=\omega^{2}\sin x$ and $W^{\prime\prime}(0)=1$, the homogenous equilibria
are $0=(0,...,0)$ and $\pi=(\pi,...,\pi)$.

Given that $U^{\prime\prime}(0)=\omega^{2}$ and $U^{\prime\prime}(\pi
)=-\omega^{2}$, the condition that gives the existence of bifurcation is
\[
\nu_{k}=\sqrt{\pm\omega^{2}+(2\sin k\pi/n)^{2}}>0\text{,}%
\]
where the minus sign correspond to $\pi$. For the $0$ equilibrium, the
frequencies $\nu_{k}$ are always positive. For the $\pi$ equilibrium, despite
the fact that the uncoupled system is unstable, the coupled system has
$2\pi/\nu_{k}$-periodic solutions near the equilibrium for $k\in\lbrack
k_{0},n/2]\cap\mathbb{N}$, where
\[
k_{0}\sim(n/\pi)\arcsin(\omega/2).
\]

In Theorem \ref{Thm1}, the non-resonant condition of $\nu_{k}$ is equivalent
to $\omega^{2}\neq\pm\omega_{l}(j)$ for integers $l\geq1$ and $j>k$, where%

\begin{equation}
\omega_{l}(j):=-\frac{(2\sin k\pi/n)^{2}-(2\sin j\pi/n)^{2}/l^{2}}{1-1/l^{2}%
}\text{.} \label{ojl}%
\end{equation}
For the $0$ equilibrium, the non-resonant condition $\omega^{2}\neq\omega
_{l}(j)$ holds true except for a finite number of resonant parameters
$\omega_{l}(j)$, and for $\pi$, except for countable number of parameters
$-\omega_{l}(j)\rightarrow2\sin k\pi/n$ as $l\rightarrow\infty$.

The same statements are true for the discrete Klein-Gordon equation with
potentials $U=\omega^{2}x^{2}+x^{3}$ and $W(x)=x^{2}/2$ in \cite{Pe}, and for
the bistable potential $U(x)=\omega^{2}(1-x^{2})^{2}/4$ for equilibria with
$a=0$ and $a=\pm1$.

\subsection{FPU and Toda latices}

In this section, we consider oscillators with $U(x)=0$ and nonlinear
interactions%
\[
W(x)=x^{2}/2+\sum_{k=3}^{\infty}\beta_{k}x^{k}\text{.}%
\]
The FPU lattice corresponds to $W(x)=x^{2}/2+\beta x^{3}/3$ and Toda lattice
to $W(x)=e^{-x}+x-1$.

Since $W^{\prime}(0)=1$, then $0$ is an homogeneous equilibrium. Therefore,
for each $k\in\lbrack1,n/2]\cap\mathbb{N}$ such that
\[
\nu_{k}=2\sin k\pi/n
\]
is non-resonant, the equilibrium has three global bifurcating branches of
periodic solutions. The non-resonant condition $\nu_{j}\neq l\nu_{k}$ is
equivalent to $\omega_{l}(j)\neq0$.

Actually, in \cite{Ri01} is proven that resonances ($\nu_{j}=l\nu_{k}$) and
higher order resonances do exist. In these cases, Theorem \ref{Thm1} only
proves the existence of the bifurcations with the higher frequency $\nu
_{j}=l\nu_{k}$. In \cite{Gu09} a two dimensional family of superposed
traveling waves is constructed for resonances ($\nu_{j}=l\nu_{k}$) but in the
context of infinite FPU lattices.

\subsection{Homogeneous lattices}

In this section we discuss the existence of standing waves when $\nu_{k}=0$
for all $k$'s. We reproduce the case $U=0$ and $W(x)=\frac{2}{5}\left\vert
x\right\vert ^{5/2}$ presented in \cite{Ja12}.

Equations (\ref{Ec}) have standing waves of the form $q_{j}(t)=a_{j}q(t)$ when%
\begin{equation}
-\ddot{q}=W^{\prime}(q) \label{DE}%
\end{equation}
and
\begin{equation}
-a_{j}=W^{\prime}(a_{j+1}-a_{j})-W^{\prime}(a_{j}-a_{j-1})\text{.} \label{ae}%
\end{equation}
Since $W(x)$ is convex, equation (\ref{DE}) has only periodic solutions.

Let $b_{j}=W^{\prime}(a_{j}-a_{j-1})$ be the momentum of $a_{j}$. Equation
(\ref{ae}) is equivalent to%
\[
-a_{j}=b_{j+1}-b_{j}\text{,}\qquad a_{j}-a_{j-1}=b_{j}\left\vert
b_{j}\right\vert ^{-1/3}\text{.}%
\]
Let $\phi:\mathbb{R}^{2}\rightarrow\mathbb{R}^{2}$ be%
\begin{equation}
\phi(a,b)=\left(  a+b\left\vert b\right\vert ^{-1/3},b-(a+b\left\vert
b\right\vert ^{-1/3})\right)  \text{.}%
\end{equation}
The orbits $(a_{j})_{j\in\mathbb{Z}}$ of $\phi$ are solutions of (\ref{ae}).

In \cite{Ja11} is shown that the map $\phi$ is conservative and has many
invariant orbits near $(0,0)$. These invariant orbits correspond to periodic
and quasiperiodic (in space) standing waves. These arguments exhibit existence
of standing waves even in the case that topological methods cannot be applied.
Therefore, further investigations are required to establish existence of
standing waves for the beads problem studied in \cite{Ja11} and \cite{StKe12}.

\section*{Appendix}

The action of $D_{n}\times O(2)$ in the irreducible representation
$(z_{1},z_{2})\in e_{1}\oplus e_{n-1}$ is given by
\begin{align}
\rho(\zeta,\varphi)(z_{1},z_{2})  &  =e^{i\varphi}(e^{i\zeta}z_{1},e^{-i\zeta
}z_{2})\text{, }\label{acction}\\
\rho(\kappa)(z_{1},z_{2})  &  =(z_{2},z_{1})\text{,}\nonumber\\
\rho(\bar{\kappa})(z_{1},z_{2})  &  =(\bar{z}_{2},\bar{z}_{1})\text{.}%
\nonumber
\end{align}

\begin{proposition}
The representation (\ref{acction}) has three maximal isotropy groups with
fixed point spaces of real dimension equal to one. The isotropy groups are:%
\[%
\begin{array}
[c]{|c|c|c|c|}\hline
\text{Parity} & \text{Orbit points} & \text{Generators} & \text{Isotropy
group}\\\hline
n=any & (r,0) & (\zeta,-\zeta),(\kappa,\bar{\kappa}) & T_{1}\\\hline
n=odd &
\begin{array}
[c]{c}%
(r,r)\\
(r,-r)
\end{array}
&
\begin{array}
[c]{c}%
(\kappa,0),(0,\bar{\kappa})\\
(\kappa,\pi),(0,\pi\bar{\kappa})
\end{array}
&
\begin{array}
[c]{c}%
S_{1}\\
\widetilde{S}_{1}%
\end{array}
\\\hline
n=even &
\begin{array}
[c]{c}%
(r,r)\\
(r,re^{i\zeta})
\end{array}
&
\begin{array}
[c]{c}%
(\kappa,0),(0,\bar{\kappa}),(\pi,\pi)\\
(\kappa\zeta,0),(0,\zeta\bar{\kappa}),(\pi,\pi)
\end{array}
&
\begin{array}
[c]{c}%
S_{1}\\
\widetilde{S}_{1}%
\end{array}
\\\hline
\end{array}
\text{.}%
\]

\end{proposition}

\begin{proof}
We need to analyze isotropy groups of orbit types. By applying $\kappa$, we
can assume that $z_{1}\neq0$, unless both coordinates are zero. Using the
action of $S^{1}$, we can assume that $(z_{1},z_{2})=(r,\rho e^{i\psi})$ with
$r>0$. Moreover, by the action
\[
\rho(l\zeta,-l\zeta)(r,\rho e^{i\psi})=(r,\rho e^{i(\psi-2l\zeta)})\text{,}%
\]
we have that $\psi\in(-\zeta,\zeta]$, and by the action of $(\kappa
,\bar{\kappa})$ that $\psi\in\lbrack0,\zeta]$. Therefore, we need to analyze
only isotropy groups of
\[
(z_{1},z_{2})=(r,\rho e^{i\psi})\text{ for }\psi\in\lbrack0,\zeta]\text{.}%
\]
Actually, for $n$ odd, the element $2\zeta$\ generate the group $\mathbb{Z}%
_{n}$ and we can take $\psi\in\lbrack0,\zeta/2]$.

The elements $(l\zeta,\varphi),(\kappa l\zeta,\varphi)\in D_{n}\times S^{1}$
act as%
\begin{align*}
\rho(l\zeta,\varphi)(r,\rho e^{i\psi})  &  =(re^{i(l\zeta+\varphi)},\rho
e^{i(\psi-l\zeta+\varphi)})\text{,}\\
\rho(\kappa l\zeta,\varphi)(r,\rho e^{i\psi})  &  =(\rho e^{i(\psi
-l\zeta+\varphi)},re^{i(l\zeta+\varphi)})\text{.}%
\end{align*}
The first coordinate $r$ is fixed by $(l\zeta,\varphi)$ if $\varphi=-l\zeta$,
and by $(\kappa l\zeta,\varphi)$ if $\varphi=l\zeta-\psi$ mod $2\pi$. These
elements act as%
\begin{align}
\rho(l\zeta,-l\zeta)(r,\rho e^{i\psi})  &  =(r,\rho e^{i(\psi-2l\zeta
)})\text{,}\label{mz}\\
\rho(\kappa l\zeta,l\zeta-\psi)(r,\rho e^{i\psi})  &  =(\rho,re^{i(2l\zeta
-\psi)})\text{.} \label{kmz}%
\end{align}
Thus, we need to find out when these elements fix the second coordinate.

For $\rho=0$, elements $(l\zeta,-l\zeta)$ always fix $(r,0)$. For $\rho\neq0$,
elements $(l\zeta,-l\zeta)$ fix $(r,\rho e^{i\psi})$ if $\psi-2l\zeta=\psi$ or
$l\zeta=\pi$ (mod $2\pi$). That is, the element $(\pi,\pi)$ is in the isotropy
group of $(r,\rho e^{i\psi})$ when $n$ is even.

For $\rho\neq r$, elements $(\kappa l\zeta,l\zeta-\psi)$ never fix these
points. For $\rho=r$, elements $(\kappa l\zeta,l\zeta-\psi)$ fix
$(r,re^{i\psi})$ when $2l\zeta-\psi=\psi$ (mod $2\pi$) or $\psi=l\zeta$ (mod
$\pi$). Then $(\kappa\psi,0)$ is in the isotropy group depending on the parity
of $n$. If $n$ is odd, the element $(\kappa,0)$ fixes $\psi=0$, and $\left(
\kappa,\pi\right)  $ fixes the point $\psi=\pi$. If $n$ is even, the element
$(\kappa,0)$ fixes $\psi=0$, and $(\kappa\zeta,0)$ fixes the point $\psi
=\zeta$.

We conclude that the orbit points in the table are fixed by the generators.
Moreover, the fixed point spaces of the isotropy groups are the set of orbit
points (in the table) for $r\in\mathbb{R}$. Therefore, in the irreducible
representation, the fixed point spaces have real dimension equal to one.
\end{proof}

The action of $D_{n}\times O(2)$ in the irreducible representation
$e_{k}\oplus e_{n-k}$ is given in Proposition \ref{Ac}. This action is similar
to (\ref{acction}), except that $\zeta\in D_{n}$ acts as
\[
\rho(\zeta)(z_{1},z_{2})=(e^{ik\zeta}z_{1},e^{-ik\zeta}z_{2})\text{.}%
\]

\begin{proposition}
For $k\in\lbrack1,n/2)\cap\mathbb{N}$, the representation $e_{k}\oplus
e_{n-k}$ has three isotropy groups with fixed point spaces of real dimension
equal to one. The isotropy groups are:%
\[%
\begin{array}
[c]{|c|c|c|c|}\hline
\text{Parity} & \text{Orbit P} & \text{Generators} & \text{Isotropy G}\\\hline
\bar{n}=any & (r,0) & (\zeta,-k\zeta),(\kappa,\bar{\kappa}),(\bar{n}\zeta,0) &
T_{k}\\\hline
\bar{n}=odd &
\begin{array}
[c]{c}%
(r,r)\\
(r,-r)
\end{array}
&
\begin{array}
[c]{c}%
(\kappa,0),(0,\bar{\kappa}),(\bar{n}\zeta,0)\\
(\kappa,\pi),(0,\pi\bar{\kappa}),(\bar{n}\zeta,0)
\end{array}
&
\begin{array}
[c]{c}%
S_{k}\\
\widetilde{S}_{k}%
\end{array}
\\\hline
\bar{n}=even &
\begin{array}
[c]{c}%
(r,r)\\
(r,re^{im\zeta})
\end{array}
&
\begin{array}
[c]{c}%
(\kappa,0),(0,\bar{\kappa}),(\frac{\bar{n}}{2}m\zeta,\pi),(\bar{n}\zeta,0)\\
(\kappa m\zeta,0),(0,m\zeta\bar{\kappa}),(\frac{\bar{n}}{2}m\zeta,\pi
),(\bar{n}\zeta,0)
\end{array}
&
\begin{array}
[c]{c}%
S_{k}\\
\widetilde{S}_{k}%
\end{array}
\\\hline
\end{array}
\text{,}%
\]
where $m\in\mathbb{N}$ is the modular inverse of $\bar{k}$ ($m\bar{k}=1$ mod
$\bar{n}$). Furthermore, the Weyl groups of these isotropy groups contain the
element $\pi\in S^{1}$ that acts multiplying by $-1$ the fixed point spaces.
\end{proposition}

\begin{proof}
The action of $\zeta\in D_{n}$ can be expressed as%
\[
\rho(\zeta)(z_{1},z_{2})=(e^{i\bar{k}(2\pi/\bar{n})}z_{1},e^{-i\bar{k}%
(2\pi/\bar{n})}z_{2}).
\]
The factor $\bar{k}$ acts as a permutation because $\bar{k}$ and $\bar{n}$ are
coprime numbers. Since $m\bar{k}=1$ modulus $\bar{n}$, then
\[
\rho(m\zeta)=(e^{i(2\pi/\bar{n})},e^{-i(2\pi/\bar{n})})\text{, }\rho\left(
\frac{\bar{n}}{2}m\zeta\right)  =-I\text{, }\rho(\bar{n}\zeta)=I\text{.}%
\]
The results in the table can be obtained using the previous proposition with
$m\zeta=2\pi/\bar{n}$ instead of $\zeta$. Furthermore, the element $\pi\in
O(2)$ leaves invariant the fixed point spaces and acts as $-1$. Then $\pi$ is
contained in the Weyl groups.
\end{proof}

The $1$-th Fourier mode $x_{1}\in\mathbb{C}^{n}$ is the direct sum of the
irreducible representations $e_{1}$, $e_{n/2}$ and $e_{k}\oplus e_{n-k}$ for
$k\in\lbrack1,n/2)\cap\mathbb{N}$.\ Let%
\[
\mathrm{Fix}(H)=\{x_{1}\in\mathbb{C}^{n}:\rho(\gamma)x_{1}=x_{1},\gamma\in
H\}.
\]

\begin{proposition}
\label{WM}For $n$ odd,%
\begin{equation}
\dim_{\mathbb{R}}\mathrm{Fix}(S)=n/2+1/2\text{,}\quad\dim_{\mathbb{R}%
}\mathrm{Fix}(\widetilde{S})=n/2-1/2\text{,}%
\end{equation}
and for $n$ even,
\begin{equation}
\dim_{\mathbb{R}}\mathrm{Fix}(S)=n/2+1\text{,}\quad\dim_{\mathbb{R}%
}\mathrm{Fix}(\widetilde{S})=n/2\text{.}%
\end{equation}

\end{proposition}

\begin{proof}
For $n$ odd, the group $S$ is generated by $(\kappa,0)$ and $(0,\bar{\kappa})$
and $\widetilde{S}$ by $(\kappa,\pi)$ and $(0,\pi\bar{\kappa})$. For
$k\in\lbrack1,n/2)\cap\mathbb{N}$, the point $(z_{1},z_{2})\in e_{k}\oplus
e_{n-k}$ is fixed by $S$ if $z_{1}=z_{2}\in\mathbb{R}$ and by $\widetilde{S}$
if $z_{1}=-z_{2}\in\mathbb{R}$. Then, the fixed point space of the groups $S$
and $\widetilde{S}$ have dimension one in $e_{k}\oplus e_{n-k}$ . We conclude
that the total dimension of the fixed point space for the representations
$k\in\lbrack1,n/2)\cap\mathbb{N}$ is $(n-1)/2$. For $k=n$, the point
$x_{n,1}\in\mathbb{C}$ is fixed by $S$ if $x_{n,1}\in\mathbb{R}$, and by
$\widetilde{S}$ if $x_{n,1}=0$.

For $n$ even, the group $\widetilde{S}$ is generated by $(\kappa\zeta,0)$ and
$(0,\zeta\bar{\kappa})$. For $k\in\lbrack1,n/2)\cap\mathbb{N}$, we have the
actions
\begin{align*}
\rho(0,\zeta\bar{\kappa})(z_{1},z_{2})  &  =(e^{i\zeta}\bar{z}_{2},e^{i\zeta
}\bar{z}_{1})\text{,}\\
\rho(\kappa\zeta,0)(z_{1},z_{2})  &  =(e^{-ik\zeta}z_{2},e^{ik\zeta}z_{1}).
\end{align*}
Thus, the point $(z_{1},z_{2})\in e_{k}\oplus e_{n-k}$ is fixed by $S$ if
$z_{1}=z_{2}\in\mathbb{R}$, and by $\widetilde{S}$ if $z_{1}=e^{i\zeta}\bar
{z}_{2}$ and $z_{1}=e^{i(k+1)\zeta}z_{2}$. Both conditions define subspaces of
dimension equal to one. Therefore, the total dimension of the fixed point
space, for $k\in\lbrack1,n/2)\cap\mathbb{N}$, is $n/2-1$. For $k=n/2,n$, the
point $x_{k,1}\in\mathbb{C}$ is fixed by $S$ if $x_{k,1}\in\mathbb{R}$.
Moreover, the representation $k=n$ is fixed by $\widetilde{S}$ if
$x_{n,1}=re^{i\zeta}$.
\end{proof}

\textbf{Acknowledgements. }C. Garc\'{\i}a is grateful to G. James for the
useful discussions about this problem, and to J. Alc\'{a}ntara, M.
Ballesteros, M. Tejada-Wriedt and the referees for their comments about the manuscript.

\end{document}